\documentclass[twoside,12pt]{article}
\usepackage{amsbsy,latexsym,amsopn,amstext,amsxtra,euscript,amscd}

\usepackage{amssymb,amsfonts,amsmath,amsthm}
\usepackage{graphics,graphicx}

\usepackage{color}

\newcommand{\seqnum}[1]{\href{http://oeis.org/#1}{\underline{#1}}}

\usepackage[table]{xcolor}
\usepackage[pdftex]{hyperref}

\newcommand*{\N}{4}
\newcommand*{\K}{6}

\newcommand*{\ADR}{University of Sopron,  Institute of Mathematics, 9400 Sopron, Bajcsy~Zs.~u.~4., Hungary. $\ \ \ \ $ \textit{nemeth.laszlo@uni-sopron.hu}}

\usepackage[all]{xy}
\usepackage{tikz}
\usetikzlibrary{arrows,chains,matrix,positioning,scopes}

\tikzset{>=stealth',every on chain/.append style={join},
         every join/.style={->}}
\tikzstyle{labeled}=[execute at begin node=$\scriptstyle,
   execute at end node=$]

\def\binomtwo#1#2{\binom{#1}{#2}_{\!\!2}}

\def\binomtwos#1#2{\left[\!\!{\binom{#1}{#2}}\!\!\right]_{\!2}}

\def\T{${\cal T}$}

\theoremstyle{plain}
\newtheorem{theorem}{Theorem}
\newtheorem{corollary}[theorem]{Corollary}

\theoremstyle{definition}

\newtheorem{remark}[theorem]{Remark}

\newcommand*{\TIT}{The trinomial transform triangle}
\title{\bf \TIT}

\author{L\'aszl\'o N\'emeth\footnote{\ADR}}
\date{}

\begin{document}

\maketitle

\begin{abstract}
	The trinomial transform of a sequence is a generalization of the well-known binomial transform, replacing binomial coefficients with trinomial coefficients.  We examine Pascal-like triangles under trinomial transform, focusing on the ternary linear recurrent sequences. We determine the sums and alternating sums of the elements in columns, and we give some examples of the trinomial transform triangle.    \\[1mm]
	The  final  publication  is  available  at \href{https://cs.uwaterloo.ca/journals/JIS/}{Journal of Integer Sequences}. 
	\end{abstract}


\section{Introduction}

The binomial transform of a sequence $(a_i)_{i=0}^\infty$ is a sequence $(b_n)_{n=0}^\infty$ defined by $b_n=\sum_{i=0}^{n}\binom{n}{i} a_i$.
This transformation is invertible with formula $a_n=\sum_{i=0}^{n}\binom{n}{i} (-1)^{n-i} b_i$. Several studies  \cite{Barbero,Nemeth_binom,Pan,Spivey} examine the properties and the generalizations of the binomial transform. 
One generation is the \textit{trinomial transform}. Let it be given by 
\begin{equation*}\label{eq:trinomial_trans}
	b_n=\sum_{i=0}^{2n}\binomtwo{n}{i} a_i,
\end{equation*}
where  for $0\leq i\leq 2n$
\begin{equation*}\label{eq:trinom}
	\binomtwo{n}{i}=\sum_{j=0}^{i}\binom{n}{j}\binom{j}{i-j}
\end{equation*}
holds with the classical binomial coefficients. It is known that the trinomial triangle (Figure~\ref{fig:trinom}) determines the trinomial coefficients  $\binomtwo{n}{i}$ which arise  in the  expansion of $(1+x+x^2)^n$. (For example, Belbachir et al.\  \cite{Bel2008} discussed  some details of the trinomial coefficients, the trinomial tringle and their generalizations.)

To the best of our knowledge, this transformation has not been previously studied or even mentioned in journals, although there are eleven sequences in OEIS \cite{Sloane} without literature, which are trinomial transforms of certain sequences. We will refer to some of them in the last section. 

In this article, we define an arithmetic structure similar to an infinite triangle, where the terms are arranged in rows and columns. Let the trinomial transform triangle \T\  be defined the following way. Row~0 consists of the terms of a given sequence $(a_k)_{k=0}^\infty$,  and any term in an other row is the sum of the three terms directly above it according to Figure~\ref{fig:construction}. The exact definition is 
\begin{eqnarray}\label{eq:defsum}
	a_0^k&=&a_{k},\nonumber\\
	a_n^k&=&a_{n-1}^{k-1}+a_{n-1}^{k}+a_{n-1}^{k+1},\quad (1\leq n\leq k). 
\end{eqnarray}

Let the $\ell^{\text{th}}$ ($\ell\geq0$) diagonal sequence of \T\ be the sequence $(a_n^{n+l})_{n=0}^\infty$. 
We shall give some properties of triangle \T\ and show that the $0^{\text{th}}$ or main diagonal sequence of \T\ is the trinomial transform sequence of  $(a_k)$. Especially,  we show that  the trinomial transform of a ternary linear recurrent sequence  is also a ternary linear recurrent sequence. Moreover, we give some properties of \T, for example, sums and partial sums in the columns. 
The author \cite{Nemeth_binom}  dealt with a generalized binomial transform triangle.

In the last section, we present some special examples for the trinomial transform triangles.

\renewcommand*{\N}{3}
\renewcommand*{\K}{6}

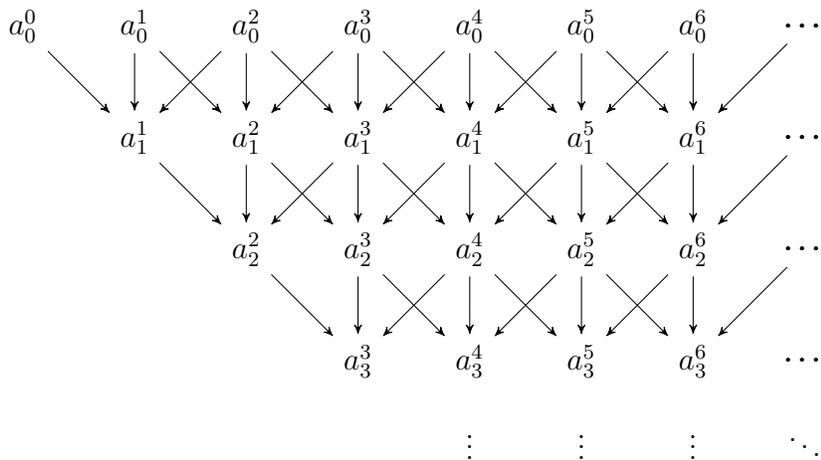
\begin{figure}[h!] \centering \scalebox{0.99}{ 		         
		\begin{tikzpicture}[->,xscale=1.5,yscale=1.5, auto,swap]
		\foreach \i in {0,...,\N}
		\foreach \j in {\i,...,\K}
		{
			\node(a\i \j)  at ({\j},{-\i}) {$a_{\i}^{\j}$};
			\node(a\i)  at ({(1+\K)},{-\i}) {$\cdots$};
		}   
		\foreach \i in {1,...,\N}
		\foreach \j in {\i,...,\K}
		{  \pgfmathtruncatemacro{\x}{(\i - 1) };
			\pgfmathtruncatemacro{\y}{\j -1 };
			\path (a\x \j) edge node {} (a\i \j) ; 
			\path (a\x \y) edge node {} (a\i \j) ;
		}
		\foreach \i in {1,...,\N}
		\foreach \j in {\i,...,{\numexpr\K-1}}
		{  \pgfmathtruncatemacro{\x}{(\i - 1) };
			\pgfmathtruncatemacro{\z}{\j +1};
			\path (a\x \z) edge node {} (a\i \j) ;
		}
		\foreach \i in {1,...,\N}
		{  \pgfmathtruncatemacro{\x}{(\i - 1) };
			\pgfmathtruncatemacro{\z}{\K +1};
			\path (a\x) edge node {} (a\i \K) ;
		}
		\pgfmathtruncatemacro{\i}{\N +1 };
		\foreach \j in {\i,...,{\numexpr\K }}
		{
			\node(a\i \j)  at ({\j},{-\i+0.3}) {$\vdots$};
		}
		\pgfmathtruncatemacro{\j}{(\K + 1) }
		\node(a\i \j)  at ({\j},{-\i+0.3}) {$\ddots$};
		\end{tikzpicture}}
	\caption{Construction of the trinomial transform triangle \T}
	\label{fig:construction}
\end{figure}

\section{Trinomial triangle and its partial sum triangle}

Most of our proofs are based on the elements of the trinomial triangle (or Feinberg's triangle named by Anatriello and Vincenzi \cite{ANVI}) and its partial sum triangle, therefore first of all we define them and give some known basic properties \cite{Andrews,Bel2008}. 

Let the trinomial coefficient $\binomtwo{n}{k}$ 
be the $k^{\text{th}}$ term of the $n^{\text{th}}$ row  in the trinomial triangle, where $0\leq n$, $0\leq k\leq 2n$ (see Figure~\ref{fig:trinom} and  \seqnum{A027907} in OEIS \cite{Sloane}). The terms of the triangle satisfy the relations
$${\binomtwo{n}{0}}={\binomtwo{n}{2n}}=1\qquad (0\leq n),$$
$${\binomtwo{n}{k}}={\binomtwo{n-1}{k-2}} +{\binomtwo{n-1}{k-1}} +\binomtwo{n-1}{k} \qquad (1\leq n).$$

We use the convention $\binomtwo{n}{k}=0$ for $k\notin\{0,1,\ldots,2n\}$.  (For more details and some generalization of binomial and trinomial coefficients see \cite{Bel2008}.)
Now, we take the partial sums of terms in columns of the trinomial triangle, let $\binomtwos{n}{k}$ denote the partial sum of values of $\binomtwo{n}{k}$ and all the terms above it in its column, so that 
\begin{equation*}
	{\binomtwos{n}{k}}= \sum_{j=0}^n {\binomtwo{n-j}{k-j}}. 
\end{equation*}
Without zero elements, we also have
\begin{equation*}
	{\binomtwos{n}{k}}= \sum_{i=|n-k|}^{n} {\binomtwo{i}{i-(n-k)}}=  \sum_{i=|n-k|}^{n} \sum_{j=0}^{i-n+k}\binom{i}{j}\binom{j}{i-n+k-j}. 
\end{equation*}

\begin{figure}[h!] \centering
	\setlength{\tabcolsep}{0pt}
	\begin{tabular}{ccccccccccccccc}
		\phantom{999}& \phantom{999}&  \phantom{999}  & \phantom{999}   & \phantom{999}   & \phantom{9999}   &  1 &  \phantom{9999}  & \phantom{999}   &  \phantom{999}  & \phantom{999}   & \phantom{999} & \phantom{999} \cr
		&   &    &    &    & 1   & 1   & 1   &    &    &    &   &   \\
		&   &    &    & 1  & 2   & 3   & 2   & 1  &    &    &   &   \\
		&   &    & 1  & 3  & 6   & 7   & 6   & 3  & 1  &    &   &   \\
		&   & 1  & 4  & 10 & 16  & 19  & 16  & 10 & 4  & 1  &   &   \\
		& 1 & 5  & 15 & 30 & 45  & 51  & 45  & 30 & 15 & 5  & 1 &   \\
		1 & 6 & 21 & 50 & 90 & 126 & 141 & 126 & 90 & 50 & 21 & 6 & 1
	\end{tabular}
	\caption{Trinomial triangle (\seqnum{A027907})}
	\label{fig:trinom}
\end{figure}

By arranging them into an infinite triangle, we gain the partial sum triangle of trinomial triangle (or partial sum trinomial triangle --- see Figure~\ref{fig:trinom_subsum}). It is not in OEIS yet, but the partial sum sequences of the central trinomial coefficients and of the two neighbouring column sequences do, see \seqnum{A097893}, \seqnum{A097861} and \seqnum{A097894}.

\begin{figure}[h!] \centering 
	\setlength{\tabcolsep}{0pt}
	\begin{tabular}{ccccccccccccccc}
		\phantom{999}& \phantom{999}&  \phantom{999}  & \phantom{999}   & \phantom{9999}   & \phantom{9999}   &  1 &  \phantom{9999}  & \phantom{9999}   &  \phantom{999}  & \phantom{999}   & \phantom{999} & \phantom{999} \cr
		&   &    &    &     & 1   & 2   & 1   &     &    &    &   &   \\
		&   &    &    & 1   & 3   & 5   & 3   & 1   &    &    &   &   \\
		&   &    & 1  & 4   & 9   & 12  & 9   & 4   & 1  &    &   &   \\
		&   & 1  & 5  & 14  & 25  & 31  & 25  & 14  & 5  & 1  &   &   \\
		& 1 & 6  & 20 & 44  & 70  & 82  & 70  & 44  & 20 & 6  & 1 &   \\
		1 & 7 & 27 & 70 & 134 & 196 & 223 & 196 & 134 & 70 & 27 & 7 & 1
	\end{tabular}
	\caption{Partial sum triangle of trinomial triangle}
	\label{fig:trinom_subsum}
\end{figure}

Our major aim in this paper is not to examine the partial sum trinomial triangle, but we give some of its properties that we use in the following sections.

From the vertical symmetry of the trinomial triangle, $\binomtwo{n}{k}=\binomtwo{n}{2n-k}$, the partial sum trinomial triangle also has a symmetry axis and 
$${\binomtwos{n}{k}}={\binomtwos{n}{2n-k}}.$$

Moreover, if $n\geq k$, then 
$${\binomtwos{n}{k}} = \sum_{i=0}^{k} {\binomtwo{n-i}{k-i}}= \sum_{i=0}^{k} \sum_{j=0}^{n-i}\binom{n-i}{j}\binom{j}{k-i-j}.$$

Furthermore, this triangle evidently satisfies the recursive relations below.
\begin{eqnarray}
	\binomtwos{n}{k}&=&\begin{cases}\label{eq:n_n1} 
		1, &\mbox{if } k =0 \mbox{ or } k=2n;\\ 
		\binomtwos{n-1}{k-1}+\binomtwo{n}{k},\qquad &  \mbox{otherwise}.
	\end{cases} \\
	\binomtwos{n}{k}&=&\begin{cases}\label{eq:nksum} 
		1, &\mbox{if } k =0 \mbox{ or } k=2n;\\ 
		n+1, & \mbox{if } k =1 \mbox{ or } k=2n-1;\\
		\binomtwos{n-1}{n-2}+\binomtwos{n-1}{n-1}+\binomtwos{n-1}{n} +1, & \mbox{if } 2\leq n=k; \\
		\binomtwos{n-1}{k-2}+\binomtwos{n-1}{k-1}+\binomtwos{n-1}{k}, & \mbox{otherwise}.
	\end{cases} 
\end{eqnarray}

Theorem~\ref{th:sum_trinom_triangle} provides summation identities for the partial sum trinomial triangle, as a corollary of the last section. 
\begin{theorem}\label{th:sum_trinom_triangle}
	The partial sum trinomial triangle satisfies the following summation expressions 
	\begin{align*}
		\sum_{k=0}^{2n} {\binomtwos{n}{k}}& \ =\ \frac{3^{n+1} - 1}{2},   &  
		\sum_{k=0}^{2n} (-1)^k{\binomtwos{n}{k}}& \ =
		\begin{cases}
			0,& \text{if $n$ is odd};\\
			1,& \text{otherwise},
		\end{cases} \\
		\sum_{k=0}^{2n} k {\binomtwos{n}{k}}  &\ =\ n\frac{3^{n+1} - 1}{2},   &  
		\sum_{k=0}^{2n} (-1)^kk {\binomtwos{n}{k}}  &\ =
		\begin{cases}
			0,& \text{if $n$ is odd or $n=0$};\\
			n,& \text{otherwise}.
		\end{cases} 
	\end{align*}
\end{theorem}

\section{Trinomial transform triangle}

Any term of \T\ can be determined by the terms of sequence $(a_k)$ with trinomial coefficients. Recall $a_0^k=a_k$.

\begin{theorem}\label{th:base}
	For any $0\leq n\leq k$, we have
	\begin{equation}\label{eq:base}
		a_n^{k}=\sum_{i=k-n}^{k+n} {\binomtwo{n}{i-k+n}} a_0^{i}. 
	\end{equation}
\end{theorem}

\begin{proof} The proof is by induction on $n$. 
	For $n=0$ the statement is trivial.
	For clarity in case $n=1$, we have 	 $$a_1^k=a_{0}^{k-1}+a_{0}^{k}+a_{0}^{k+1} =\binomtwo{1}{0}a_{0}^{k-1}+\binomtwo{1}{1}a_{0}^{k}+\binomtwo{1}{2}a_{0}^{k+1}  =\sum_{i=k-1}^{k+1} {\binomtwo{1}{i-k+1}} a_0^{i}.$$
	Let us suppose that the result is true for $n-1$, so, for example,
	$$a_{n-1}^{k}=\sum_{i=k-n+1}^{k+n-1}{\binomtwo{n-1}{i-k+n-1}} a_0^{i}.$$
	\noindent Using \eqref{eq:nksum}, we have 
	\begin{eqnarray*}
		a_{n}^{k}&=&a_{n-1}^{k-1}+a_{n-1}^{k}+a_{n-1}^{k+1}\\
		&=&\sum_{i=k-n}^{k+n-2}{\binomtwo{n-1}{i-k+n}} a_0^{i} +\sum_{i=k-n+1}^{k+n-1}{\binomtwo{n-1}{i-k+n-1}} a_0^{i} +\sum_{i=k-n+2}^{k+n}{\binomtwo{n-1}{i-k+n-2}} a_0^{i}\\
		&=& {\binomtwo{n-1}{0}} a_0^{k-n} +\left({\binomtwo{n-1}{1}} +\binomtwo{n-1}{0} \right)  a_0^{k-n+1} \\
		& & + \left({\binomtwo{n-1}{2}} +{\binomtwo{n-1}{1}} +\binomtwo{n-1}{0} \right) a_0^{k-n+2} +\cdots \\
		& & + \left({\binomtwo{n-1}{i-k+n-2}} +{\binomtwo{n-1}{i-k+n-1}} +\binomtwo{n-1}{i-k+n} \right) a_0^{i} +\cdots \\
		& & + \left({\binomtwo{n-1}{2n-3}} +\binomtwo{n-1}{2n-2} \right)  a_0^{k+n-1} + {\binomtwo{n-1}{2n-2}} a_0^{k+n}\\
		&=& \sum_{i=k-n}^{k+n} {\binomtwo{n}{i-k+n}} a_0^{i}.
	\end{eqnarray*}
\end{proof}

We let $(b_n)$ denote the  $0^{\text{th}}$ diagonal sequence  $(a_n^n)$ in \T.
If $k=n$, then we obtain  the next corollary from Theorem~\ref{th:base}.
\begin{corollary}
	The diagonal sequence $(b_n)$ of the trinomial transform triangle is the trinomial transform sequence of $(a_k)$.
\end{corollary}

Let  $(s_n)_{n=0}^\infty$ be the sum sequence of the values of columns in \T, so that
\begin{equation}\label{eq:sumrowdef}
	s_n=\sum_{i=0}^{n}a_i^n.
\end{equation}

\begin{theorem} 
	\begin{equation*}\label{eq:sumrowtri}
		s_n=\sum_{i=0}^{n} \sum_{j=n-i}^{n+i} {\binomtwo{i}{j-n+i}} a_0^{j}
		=\sum_{i=0}^{n} \sum_{k=0}^{2i} {\binomtwo{i}{k}} a_0^{n+k-i}.
	\end{equation*}
\end{theorem}
\begin{proof}
	Substitute \eqref{eq:base} into  \eqref{eq:sumrowdef} and put $k=i+j-n$.
\end{proof}

Let us express the term $s_n$ by the help of the terms of the partial sum  trinomial triangle and for it, first, we have to prove the theorem below. 
\begin{theorem}\label{th:sorsum}
	If $n\leq k$, then
	\begin{equation}
		\sum_{i=0}^{n} a_i^{k}=\sum_{\ell=k-n}^{k+n} {\binomtwos{n}{\ell-k+n}} a_0^{\ell}. 
	\end{equation}
\end{theorem}
\begin{proof}
	We prove again by induction. The case $n=0$ is clear, and for clarity in case $n=1$, we have $\sum_{i=0}^{1} a_i^{k}=a_0^{k}+a_1^{k}=a_0^{k-1}+2a_0^{k}+a_0^{k+1}$.
	The hypothesis of induction for $n-1$ is
	$$\sum_{i=0}^{n-1} a_i^{k}=\sum_{\ell=k-n+1}^{k+n-1} {\binomtwos{n-1}{\ell-k+n+1}} a_0^{\ell}.$$
	
	Now, using \eqref{eq:n_n1}, we obtain 
	\begin{eqnarray*}
		\sum_{i=0}^{n} a_i^{k}\!&=&\sum_{i=0}^{n-1} a_i^{k}+ a_{n}^{k}\\
		&=&\sum_{\ell=k-n+1}^{k+n-1} {\binomtwos{n-1}{\ell-k+n+1}} a_0^{\ell}+ \sum_{i=k-n}^{k+n} {\binomtwo{n}{i-k+n}} a_0^{i}\\
		&=&\sum_{\ell=k-n+1}^{k+n-1} {\binomtwos{n-1}{\ell-k+n+1}} a_0^{\ell}+  \binomtwo{n}{0} a_0^{k-n} +\binomtwo{n}{2n} a_0^{k+n}+  \sum_{\ell=k-n+1}^{k+n-1} {\binomtwo{n}{\ell-k+n}} a_0^{\ell}\\
		&=& \binomtwo{n}{0} a_0^{k-n}+
		\sum_{\ell=k-n+1}^{k+n-1}\left( {\binomtwos{n-1}{\ell-k+n+1}}+\binomtwo{n}{\ell-k+n}\right) a_0^{\ell}
		+\binomtwo{n}{2n} a_0^{k+n}\\
		&=& \sum_{\ell=k-n}^{k+n} {\binomtwos{n}{\ell-k+n}} a_0^{\ell}.
	\end{eqnarray*}
\end{proof}

The same induction method yields the following theorem.
\begin{theorem} If $0\leq j\leq n$, then we have
	\begin{equation*}
		\sum_{i=0}^{k} a_i^{k}=\sum_{\ell=k-n+j}^{k+n-j} {\binomtwos{n-j}{\ell-k+n-j}} a_j^{\ell}.
	\end{equation*}
\end{theorem}

Applying Theorem \ref{th:sorsum} for the  case $k=n$, we obtain the main theorem of this section. 
\begin{theorem} The column sum sequence can be given with the  expression
	\begin{equation*}
		s_{n} =\sum_{\ell=0}^{2n} {\binomtwos{n}{\ell}} a_0^{\ell}.
	\end{equation*}
\end{theorem}

\section{Trinomial transform triangles generated by ternary homogeneous linear recurrent sequences}

From this point on we examine the case $a_0^k = a_k$, where $(a_k)_{k=0}^{\infty}$ is a ternary linear recursive sequence with initial values $a_0, a_1,a_2\in \mathbb{Z}$ ($|a_0|+|a_1|+|a_2|\ne 0 $). And it is defined for $3\leq k$ by
\begin{equation}\label{eq:def_ternary_seq}
	a_k=\alpha a_{k-1}+\beta a_{k-2}+\gamma a_{k-3}, 
\end{equation}
where $\alpha, \beta, \gamma\in \mathbb{Z}$,\  $\gamma \ne0$. (In general, the results of this section hold for an integral domain.)

From now on, we will use three important variables
\begin{eqnarray*}
	\mathcal{A}&=&\alpha^2+\alpha+2\beta+3,\\
	\mathcal{B}&=&-2\alpha^2+\alpha\beta+2\alpha\gamma-\beta^2-2\alpha-3\beta+3\gamma-3,\\
	\mathcal{C}&=&\alpha^2-\alpha\beta-\alpha\gamma+\beta^2-\beta\gamma+\gamma^2+\alpha+\beta-2\gamma+1.	
\end{eqnarray*}

\begin{theorem}The terms in the rows satisfy the same ternary relation,
	\begin{equation}\label{eq:rec_abc}
		a_n^k=\alpha a_{n}^{k-1}+\beta a_{n}^{k-2}+\gamma a_{n}^{k-3}  \qquad (0\leq n\leq k-3).
	\end{equation}
\end{theorem}
\begin{proof} We prove by induction on $n$. 
	If $n=0$, then \eqref{eq:rec_abc} is the definition of $a_k$. We suppose that it holds for up to row $n-1$. Then
	\begin{eqnarray*}
		a_n^k&=&a_{n-1}^{k-1}+a_{n-1}^{k}+a_{n-1}^{k+1}\\
		&=& \alpha a_{n-1}^{k-2}+\beta a_{n-1}^{k-3}+\gamma a_{n-1}^{k-4}
		+\alpha a_{n-1}^{k-1}+\beta a_{n-1}^{k-2}+\gamma a_{n-1}^{k-3}
		+\alpha a_{n-1}^{k}+\beta a_{n-1}^{k-1}+\gamma a_{n-1}^{k-2}\\
		&=& \alpha(a_{n-1}^{k}+a_{n-1}^{k-1}+a_{n-1}^{k-2})
		+\beta(a_{n-1}^{k-1}+a_{n-1}^{k-2}+a_{n-1}^{k-3})
		+\gamma(a_{n-1}^{k-2}+a_{n-1}^{k-3}+a_{n-1}^{k-4})\\
		&=& \alpha a_{n}^{k-1}+\beta a_{n}^{k-2}+\gamma a_{n}^{k-3}.
	\end{eqnarray*}
\end{proof}

\begin{remark}
	Let us extend the sequence $(a_k)$ for negative indices, so that let $a_k=0$, if $k<0$. From definition \eqref{eq:defsum}, the initial values of sequences of rows are $a_n^{-n}=a_0$, $a_n^{-n+1}=na_0+a_1$, and $a_n^{-n+2}=\binom{n+1}{2}a_0+na_1+a_2$ with extension for negative indices. It can be easily proved by induction.  
\end{remark}

\begin{theorem}\label{th:rec_ABC}
	The terms in the diagonals can be described by the same ternary recurrence relation,
	\begin{equation}\label{eq:rec_ABC}
		a_n^{n+\ell}=\mathcal{A}a_{n-1}^{n+\ell-1}+\mathcal{B} a_{n-2}^{n+\ell-2}+\mathcal{C} a_{n-3}^{n+\ell-3} \qquad (3\leq n,\  0\leq\ell).
	\end{equation}
\end{theorem}
\begin{proof} We prove it by induction first on $\ell$ and second on $n$.
	
	First, let $n=3$. In case $\ell=0, 1, 3$,  we can check the recurrence \eqref{eq:rec_ABC} by computer. (For example the expression of $a_3^3$ contains more than hundred characters. Figure~\ref{fig:steps_trinomial} shows a small part of the triangle based on initial elements $x$, $y$ and $z$.) 
	Now, we suppose that \eqref{eq:rec_ABC} holds for up to $\ell-1$. We obtain
	\begin{eqnarray*}
		a_3^{6+\ell}&=&\alpha a_3^{5+\ell}+\beta a_3^{4+\ell}+\gamma a_3^{3+\ell}\\
		&=&	\alpha\left(\mathcal{A}a_{2}^{4+\ell}+\mathcal{B} a_{1}^{3+\ell}+\mathcal{C} a_{0}^{2+\ell} \right) +
		\beta\left(\mathcal{A}a_{2}^{3+\ell}+\mathcal{B} a_{1}^{2+\ell}+\mathcal{C} a_{0}^{1+\ell} \right) \\&& \qquad +\ 
		\gamma\left(\mathcal{A}a_{2}^{2+\ell}+\mathcal{B} a_{1}^{1+\ell}+\mathcal{C} a_{0}^{\ell} \right)\\
		&=&	\mathcal{A}\left(\alpha a_{2}^{4+\ell }+ \beta a_{2}^{3+\ell} + \gamma a_{2}^{2+\ell} \right) +
		\mathcal{B}\left(\alpha a_{1}^{3+\ell }+ \beta a_{1}^{2+\ell} + \gamma a_{1}^{1+\ell} \right) \\&& \qquad +\ 
		\mathcal{C}\left(\alpha a_{0}^{2+\ell }+ \beta a_{0}^{1+\ell} + \gamma a_{0}^{\ell} \right) \\
		&=&  \mathcal{A}a_2^{5+\ell}+\mathcal{B}a_1^{4+\ell}+\mathcal{C}a_0^{3+\ell}.
	\end{eqnarray*}

	Second, we suppose in case any $\ell$ that \eqref{eq:rec_ABC} holds for up to $4\leq n-1$, thus we have
	\begin{eqnarray*}
		a_n^{n+\ell}&=&a_{n-1}^{n+\ell-1}+a_{n-1}^{n+\ell}+a_{n-1}^{n+\ell+1}\\
		&=&\mathcal{A}a_{n-2}^{n+\ell-2}+\mathcal{B} a_{n-3}^{n+\ell-3}+\mathcal{C} a_{n-4}^{n+\ell-4}+
		\mathcal{A}a_{n-2}^{n+\ell-1}+\mathcal{B} a_{n-3}^{n+\ell-2}+\mathcal{C} a_{n-4}^{n+\ell-3}\\
		&& \qquad +\ 	\mathcal{A}a_{n-2}^{n+\ell}+\mathcal{B} a_{n-3}^{n+\ell-1}+\mathcal{C} a_{n-4}^{n+\ell-2}\\
		&=&\mathcal{A}\left(a_{n-2}^{n+\ell-2}+ a_{n-2}^{n+\ell-1}+a_{n-2}^{n+\ell}\right)+
		\mathcal{B}\left(a_{n-3}^{n+\ell-3}+a_{n-3}^{n+\ell-2}+a_{n-3}^{n+\ell-1}\right)\\ 
		&& \qquad +\ 
		\mathcal{C}\left(a_{n-4}^{n+\ell-4}+ a_{n-4}^{n+\ell-3}+ a_{n-4}^{n+\ell-2} \right)\\
		&=&\mathcal{A}a_{n-1}^{n+\ell-1}+\mathcal{B}a_{n-2}^{n+\ell-2}+\mathcal{C} a_{n-3}^{n+\ell-3}.
	\end{eqnarray*}
\end{proof}

\begin{figure}[h!] \centering
	\scalebox{0.99}{ \begin{tikzpicture}[->,xscale=3.0,yscale=2.5, auto,swap,every node/.style={shape=rectangle,draw=none}]
		\node(a0) at (-2,-0)   {$x$};
		\node (a1) at (-1,0)   {$y$};
		\node (a2) at (0,0)   {$z$};
		\node (a3) at (1,0)   {$ \gamma x+\beta y+\alpha z $};
		\node (a4) at (2,0) [align=center]{$ \alpha\gamma x+$\\$(\alpha\beta+\gamma) y+$\\$(\alpha^2+\beta) z $};
		\node (b1) at (-1,-1)   {$x+y+z$};
		\node (b2) at (0,-1)  [ align=center]{$ \gamma x+(\beta+1)y$\\$+(\alpha+1)z$};
		\node (b3) at (1,-1) [ align=center]   {$(\alpha\gamma+\gamma)x+$\\$(\alpha\beta+\beta+\gamma)y+$\\$(\alpha^2+\alpha+\beta+1)z$};
		\node (c2) at (0,-2)  [align=center] {$(\alpha\gamma+2\gamma+1)x+$\\$(\alpha\beta+2\beta+\gamma+2)y+$\\$(\alpha^2+2\alpha+\beta+3)z$};
		\path (a0) edge node[draw=none] {} (b1) [thick];
		\path (a1) edge node[draw=none] {} (b1) [thick];
		\path (a1) edge node[draw=none] {} (b2) [thick];
		\path (a2) edge node[draw=none] {} (b1) [thick];
		\path (a2) edge node[draw=none] {} (b2) [thick];
		\path (a2) edge node[draw=none] {} (b3) [thick];
		\path (a3) edge node[draw=none] {} (b2) [thick];
		\path (a3) edge node[draw=none] {} (b3) [thick];
		\path (a4) edge node[draw=none] {} (b3) [thick];
		\path (b1) edge node[draw=none] {} (c2)[thick];
		\path (b2) edge node[draw=none] {} (c2)[thick];
		\path (b3) edge node[draw=none] {} (c2)[thick];	
		\end{tikzpicture}}
	\caption{A part of the growing of \T}
	\label{fig:steps_trinomial}
\end{figure}
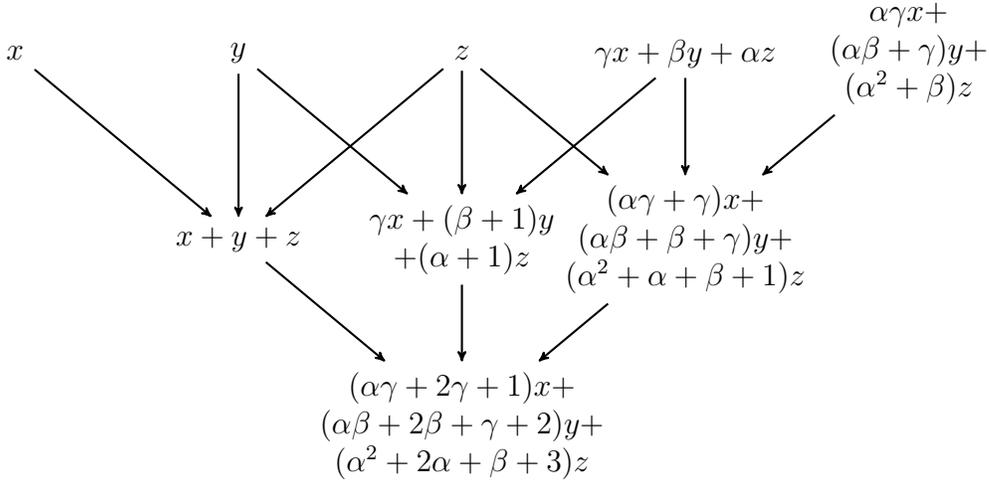

Theorem \ref{th:rec_ABC} for case $\ell=0$ and Figure \ref{fig:steps_trinomial} provide the main statement of this section.
\begin{corollary}[Main corollary]\label{th:trasform_of_ternary}
	The trinomial transform sequence of a given ternary linear recurrent sequence $(a_k)$ defined by \eqref{eq:def_ternary_seq} is the ternary linear recurrent sequence $(b_k)$ defined by 
	\begin{equation*}
		b_n=\mathcal{A}b_{n-1}+\mathcal{B} b_{n-2}+\mathcal{C} b_{n-3}  \qquad (3\leq n,\  0\leq\ell),
	\end{equation*}
	with initial values $b_0=a_0$, $b_1=a_0+a_1+a_2$ and $b_2=(\alpha\gamma+2\gamma+1)a_0+(\alpha\beta+2\beta+\gamma+2)a_1+(\alpha^2+2\alpha+\beta+3)a_2$.
\end{corollary}

Moreover, we can give a general statement for the linear homogenous recurrence sequences according to Barbero~et~al.\ \cite{Barbero1}.

\begin{theorem}
	Let $(a_k)$  be a linear recurrent sequence of degree $r$ with characteristic polynomial given by $p(t)$, and zeros $\omega_1,\omega_2,\ldots,\omega_r$ are all different. Then its trinomial transform sequence $(b_k)$  is a linear recurrent sequence of degree $r$, and zeros of the characteristic polynomial of $(b_k)$ are $\omega_1^2+\omega_1+1$, $\omega_2^2+\omega_2+1$, \ldots, $\omega_r^2+\omega_r+1$.
	\begin{proof}
		Our proof is similar to the proof of Barbero~et~al.\  \cite[Thm.\ 10]{Barbero1}. If $p(t)$ has distinct complex zeros $\omega_1,\ldots,\omega_r$, then by Binet formula
		$$ a_n=\sum_{j=1}^{r}A_j\omega_{j}^{n} , $$
		for some complex $A_i$ derived from initial conditions \cite[Thm.\ C.1.\ p.\ 33]{sho}. For all terms of $(b_k)$, we have 
		$$ b_n=\sum_{i=0}^{2n}\binomtwo{n}{i} a_i
		=\sum_{i=0}^{2n}\binomtwo{n}{i} \sum_{j=1}^{r}A_j\omega_{j}^{i}
		= \sum_{j=1}^{r}A_j(\omega_j^2+\omega_j+1)^n\,. $$
	\end{proof}
\end{theorem}

It is surprising that a ternary recurrence relation with rational coefficients holds for the finite column sequences $(a_{i}^{k})_{i=0}^k$. The following theorem formulates it precisely.
\begin{theorem}\label{th:column}
	The terms in the column $k$ can be described by the ternary recurrence relation
	\begin{equation*}
		a_n^{k}=\frac{\mathcal{P}}{\gamma}a_{n-1}^{k}+\frac{\mathcal{Q}}{\gamma} a_{n-2}^{k}+\frac{\mathcal{C}}{\gamma} a_{n-3}^{k} \qquad (3\leq n\leq k),
	\end{equation*}
	where $\mathcal{P}= \alpha\gamma-\beta+3\gamma$ and 
	$\mathcal{Q}= \alpha\beta-2\alpha\gamma+\beta\gamma-\alpha+2\beta$.
\end{theorem}
\begin{proof}
	The proof is very similar to that of Theorem \ref{th:rec_ABC}. We use the induction method. 
	
	First, let $n=3$. For cases $k=3, 4, 5$ we can check the equation 
	\begin{equation}\label{eq:g3_abc}
		{\gamma} a_3^{k}=\mathcal{P}a_{2}^{k}+\mathcal{Q} a_{1}^{k}+\mathcal{C} a_{0}^{k}.
	\end{equation}
	Now, we suppose that the equation \eqref{eq:g3_abc} holds up to $k-1$. So the condition of induction yields the equations  
	\begin{eqnarray*}
		\alpha(	{\gamma} a_3^{k-1})&=& \alpha(\mathcal{P}a_{2}^{k-1}+\mathcal{Q} a_{1}^{k-1}+\mathcal{C} a_{0}^{k-1}),\\
		\beta(	{\gamma} a_3^{k-2})&=&\beta(\mathcal{P}a_{2}^{k-2}+\mathcal{Q} a_{1}^{k-2}+\mathcal{C} a_{0}^{k-2}),\\
		\gamma(	{\gamma} a_3^{k-3})&=&\gamma(	\mathcal{P}a_{2}^{k-3}+\mathcal{Q} a_{1}^{k-3}+\mathcal{C} a_{0}^{k-3}).
	\end{eqnarray*}
	Summing the equations and applying Theorem~\ref{eq:rec_abc} we obtain that \eqref{eq:g3_abc} holds for any $k$.
	
	Second, if we suppose that 
	\begin{equation}\label{eq:gk_abc}
		{\gamma} a_{n-1}^{k}=\mathcal{P}a_{n-2}^{k}+\mathcal{Q} a_{n-3}^{k}+\mathcal{C} a_{n-4}^{k}
	\end{equation} 
	holds for any $4\leq n\leq k$, then  we have
	\begin{eqnarray*}
		{\gamma} a_{n-1}^{k-1}&=& \mathcal{P}a_{n-2}^{k-1}+ \mathcal{Q} a_{n-3}^{k-1}+ \mathcal{C} a_{n-4}^{k-1},\\
		{\gamma} a_{n-1}^{k}&=& \mathcal{P}a_{n-2}^{k}+ \mathcal{Q} a_{n-3}^{k}+ \mathcal{C} a_{n-4}^{k},\\
		{\gamma} a_{n-1}^{k+1}&=& \mathcal{P}a_{n+1}^{k-3}+ \mathcal{Q} a_{n+1}^{k-3}+ \mathcal{C} a_{n-4}^{k+1}.
	\end{eqnarray*} 
	Using  \eqref{eq:defsum} for the sum of the equations  we obtain
	$${\gamma} a_{n}^{k}=\mathcal{P}a_{n-1}^{k}+\mathcal{Q} a_{n-2}^{k}+\mathcal{C} a_{n-3}^{k}.$$
\end{proof}

\subsection{Sums and alternating sums of columns}

We give $(s_n)$ as a linear recurrent sequence with coefficients given by the coefficients of sequences $(a_n)$ and $(b_n)$. Recall  $(s_n)$ was defined in \eqref{eq:sumrowdef}.

\begin{theorem}
	If $n\geq 6$, then the sequence $(s_n)$ can be described by the sixth order  homogeneous linear recurrence relation 
	\begin{multline}\label{eq:sn}
		s_{n} = (\alpha+\mathcal{A})s_{n-1}
		+(\mathcal{B}-\alpha \mathcal{A}+\beta)s_{n-2}
		-(\alpha \mathcal{B}+\beta \mathcal{A}-\mathcal{C}-\gamma)s_{n-3}\\
		-(\alpha \mathcal{C}+\beta \mathcal{B}+\gamma\mathcal{A})s_{n-4}
		-(\beta\mathcal{C}+\gamma\mathcal{B})s_{n-5}-\gamma\mathcal{C} s_{n-6}.
	\end{multline}
\end{theorem}
\begin{proof}
	First, we prove that the recurrence relations of sequences $(a_n^k)$, $(a_n^{n+\ell})$ $(\ell\geq0)$ are the same as relation \eqref{eq:sn}.
	Extend relation \eqref{eq:rec_abc} to sixth order  homogeneous linear recurrence relation by sum of $a_n^k$, $\mathcal{A}a_n^{k-1}$,  $\mathcal{B}a_n^{k-2}$ and $\mathcal{C}a_n^{k-3}$, moreover extend \eqref{eq:rec_ABC} also by sum of $a_n^{n+\ell}$, $\alpha a_{n-1}^{n+\ell-1}$,  $\beta a_{n-2}^{n+\ell-2}$ and $\gamma a_{n-3}^{n+\ell-3}$. Then we obtain the same coefficients as in \eqref{eq:sn}.
	
	Second, we give the connection between  $s_n$ and $s_{n-2}$. It is from  \begin{eqnarray*}
		s_{n-1}-a_0^{n-1}&=&\sum_{i=1}^{n-1}a_i^{n-1}=\sum_{i=1}^{n-1}\left(a_{i-1}^{n-2}+a_{i-1}^{n-1}+a_{i-1}^{n}\right)=\sum_{i=0}^{n-2}a_{i}^{n-2}+ \sum_{i=0}^{n-2}a_{i}^{n-1}+ \sum_{i=0}^{n-2}a_{i}^{n}\\
		&=&s_{n-2}+s_{n-1}-a_{n-1}^{n-1}+s_{n}-a_{n-1}^{n}-a_{n}^{n}.
	\end{eqnarray*} 
	Reordering the equation, we have 
	\begin{equation*}
		s_{n}=-s_{n-2}+a_{n-1}^{n-1}+a_{n-1}^{n}+a_{n}^{n}-a_0^{n-1}.
	\end{equation*}
	
	Third, we have finished the preparations to prove the theorem by induction on $n$. Because of the complexity in case $n=6$, we can check the recurrence \eqref{eq:sn} by computer. Then we suppose that \eqref{eq:sn} holds for cases up to $n-1$. We obtain 
	\begin{eqnarray*}
		s_{n}&=& -s_{n-2}+a_{n-1}^{n-1}+a_{n-1}^{n}+a_{n}^{n}-a_0^{n-1}\\
		&=& 
		-((\alpha+\mathcal{A})s_{n-3}
		-(\mathcal{B}-\alpha \mathcal{A}+\beta)s_{n-4}
		-(\alpha \mathcal{B}+\beta \mathcal{A}-\mathcal{C}-\gamma)s_{n-5}\\&& 
		\quad-(\alpha \mathcal{C}+\beta \mathcal{B}+\gamma\mathcal{A})s_{n-6}
		-(\beta\mathcal{C}+\gamma\mathcal{B})s_{n-7}
		-\gamma\mathcal{C} s_{n-8})\\&& 
		+(\alpha+\mathcal{A})a_{n-2}^{n-2}
		+(\mathcal{B}-\alpha \mathcal{A}+\beta)a_{n-3}^{n-3}
		+(\alpha \mathcal{B}+\beta \mathcal{A}-\mathcal{C}-\gamma)a_{n-4}^{n-4}\\ 	&& 
		\quad+(\alpha \mathcal{C}+\beta \mathcal{B}+\gamma\mathcal{A})a_{n-5}^{n-5}
		+(\beta\mathcal{C}+\gamma\mathcal{B})a_{n-6}^{n-6}
		+\gamma\mathcal{C} a_{n-7}^{n-7}
		\\&&  +\ \cdots \ \\	&& 
		-(\alpha+\mathcal{A})a_{0}^{n-2}
		-(\mathcal{B}-\alpha \mathcal{A}+\beta)a_0^{n-3}
		-(\alpha \mathcal{B}+\beta \mathcal{A}-\mathcal{C}-\gamma)a_0^{n-4}\\&& 
		\quad-(\alpha \mathcal{C}+\beta \mathcal{B}+\gamma\mathcal{A})a_0^{n-5}
		-(\beta\mathcal{C}+\gamma\mathcal{B})a_0^{n-6}
		-\gamma\mathcal{C} a_0^{n-7}\\
		&=& (\alpha+\mathcal{A})(-s_{n-3}+a_{n-2}^{n-2}+a_{n-2}^{n-1}+a_{n-1}^{n-1}-a_0^{n-2})\\
		&&+\ (\mathcal{B}-\alpha \mathcal{A}+\beta)(-s_{n-4}+a_{n-3}^{n-3}+a_{n-3}^{n-2}+a_{n-2}^{n-2}- a_0^{n-3})\\
		&&-\ (\alpha \mathcal{B}+\beta \mathcal{A}-\mathcal{C}-\gamma)(-s_{n-5}+a_{n-4}^{n-4}+a_{n-4}^{n-3}+ a_{n-3}^{n-3}-a_0^{n-4})\\
		&&-\ (\alpha \mathcal{C}+\beta \mathcal{B}+\gamma\mathcal{A})(-s_{n-6}+a_{n-5}^{n-5}+a_{n-5}^{n-4}+a_{n-4}^{n-4}- a_0^{n-5})\\
		&&-\ (\beta\mathcal{C}+\gamma\mathcal{B})(-s_{n-7}+a_{n-6}^{n-6}+a_{n-6}^{n-5}+ a_{n-5}^{n-5}-a_0^{n-6})\\
		&&-\ \gamma\mathcal{C} (-s_{n-8}+a_{n-7}^{n-7}+a_{n-7}^{n-6}+ a_{n-6}^{n-6}- a_0^{n-7}) \\
		&=&  (\alpha+\mathcal{A})s_{n-1}
		+(\mathcal{B}-\alpha \mathcal{A}+\beta)s_{n-2}
		-(\alpha \mathcal{B}+\beta \mathcal{A}-\mathcal{C}-\gamma)s_{n-3}\\
		&&	-(\alpha \mathcal{C}+\beta \mathcal{B}+\gamma\mathcal{A})s_{n-4}
		-(\beta\mathcal{C}+\gamma\mathcal{B})s_{n-5}-\gamma\mathcal{C} s_{n-6}.
	\end{eqnarray*}
\end{proof}

Let  $(\bar{s}_n)$  be the alternating sum of the values of columns $n$, so that
\begin{equation*}\label{eq:altsumrow}
	\bar{s}_n=\sum_{i=0}^{n}(-1)^ia_i^n=\sum_{i=0}^{n}  \sum_{k=0}^{2i} (-1)^i {\binomtwo{i}{k}} a_0^{n+k-i}.
\end{equation*}	

\begin{theorem}	
	If $n\geq 6$, then the alternating sum sequences $(\bar{s}_n)$ can be described by the sixth order  homogeneous linear recurrence relation 
	\begin{multline*}\label{eq:altsn}
		\bar{s}_{n} = (\alpha-\mathcal{A})\bar{s}_{n-1}
		+(\mathcal{B}+\alpha \mathcal{A}+\beta)\bar{s}_{n-2}
		-(\alpha \mathcal{B}-\beta \mathcal{A}+\mathcal{C}-\gamma)\bar{s}_{n-3}\\
		+(\alpha \mathcal{C}-\beta \mathcal{B}+\gamma\mathcal{A})\bar{s}_{n-4}
		+(\beta\mathcal{C}-\gamma\mathcal{B})\bar{s}_{n-5}+ \gamma\mathcal{C} \bar{s}_{n-6}.
	\end{multline*}	
\end{theorem}
\begin{proof}
	Row by row the signs of the terms in the alternating sums change in directions parallel to the diagonal, hence we have to change the sign of $\mathcal{A}$ and $\mathcal{C}$ in the summation relation \eqref{eq:sn}. They have influence on the signs of every next and third terms, respectively. In a row the signs of the terms do not change.
\end{proof}

\section{Examples}

In this section, we give some examples for the trinomial transform triangles generated by ternary linear recurrent sequences.

\subsection{Fibonacci sequence}

Let the base sequence be the Fibonacci sequence $(F_k)_{k=0}^\infty$  defined by $F_0=0$, $F_1=1$ and $F_k=F_{k-1}+F_{k-2}$, $k\geq 2$ (\seqnum{A000045}). If we extend it to a ternary recurrence (with sum of $F_k$ and $F_{k-1}$), then we obtain that $F_k=2F_{k-1}-F_{k-3}$,  $k\geq 3$. So let $a_0^k=F_k$ for any $k$, moreover $\alpha=2$, $\beta=0$ and $\gamma=-1$.

When we substitute the initial values into  \T, then we have the Fibonacci trinomial transform triangle, see Table~\ref{tab:fib}.

\begin{table}[htb]
	\centering
	\setlength{\tabcolsep}{3pt}
	\begin{tabular}{|c|cccccccccc|}\hline
		& 0           & 1          & 2          & 3           & 4            & 5            & 6             & 7             & 8              & 9                          \\ \hline
		0         & 0          & 1          & 1           & 2           & 3            & 5             & 8              & 13             & 21              & 34      \\
		1         &            & 2          & 4           & 6           & 10           & 16            & 26             & 42             & 68              & 110     \\
		2         &            &            & 12          & 20          & 32           & 52            & 84             & 136            & 220             & 356     \\
		3         &            &            &             & 64          & 104          & 168           & 272            & 440            & 712             & 1152    \\
		4         &            &            &             &             & 336          & 544           & 880            & 1424           & 2304            & 3728    \\
		5         &            &            &             &             &              & 1760          & 2848           & 4608           & 7456            & 12064   \\
		6         &            &            &             &             &              &               & 9216           & 14912          & 24128           & 39040   \\
		7         &            &            &             &             &              &               &                & 48256          & 78080           & 126336  \\
		8         &            &            &             &             &              &               &                &                & 252672          & 408832  \\
		9         &            &            &             &             &              &               &                &                &                 & 1323008 \\ \hline
		$s_n$  & \textbf{0} & \textbf{3} & \textbf{17} & \textbf{92} & \textbf{485} & \textbf{2545} & \textbf{13334} & \textbf{69831} & \textbf{365661} & \textbf{1914660}\\
		$\bar{s}_n$ & 0 & $-1$ & 9 & $-48$ & 257 & $-1343$& 7042 & $-36861$ & 193029 & $-1010680$\\ \hline
	\end{tabular}
	\caption{Fibonacci trinomial transform triangle} \label{tab:fib}
\end{table}

Having observed the rows of Table~\ref{tab:fib}, we have found that they are the well-known \mbox{$i$-Fibonacci} sequences, defined by $F^{[i]}_0=0$, $F^{[i]}_1=i$ and $F^{[i]}_j=F^{[i]}_{j-1}+F^{[i]}_{j-2}$, $j\geq 2$. This result is expressed in the following theorem.

\begin{theorem}\label{th:fib01}
	The terms of the rows in Table~\ref{tab:fib} are the terms of the $2^n$-Fibonacci sequences, so that
	\begin{equation*}\label{eq:fib_af}
		a_n^k= F^{[2^n]}_{k+n}.
	\end{equation*} 
\end{theorem}
\begin{proof}
	Obviously, for case $n=1$ the theorem is true. Let us suppose that the result is also true for $n-1$. As the relation between the Fibonacci and $i$-Fibonacci sequences is $i\cdot F_k=F^{[i]}_k$, we gain
	\begin{eqnarray*}
		a_n^k&=& a_{n-1}^{k-1}+a_{n-1}^{k}+a_{n-1}^{k+1}
		=F^{[2^{n-1}]}_{k-1+n-1}+F^{[2^{n-1}]}_{k+n-1}+F^{[2^{n-1}]}_{k+1+n-1)}\\
		&=& 2 F^{[2^{n-1}]}_{k-n}= F^{[2^{n}]}_{k-n}.
	\end{eqnarray*}
\end{proof}

For all the directions parallel to the diagonal of Table~\ref{tab:fib} and for the trinomial transform sequence, we obtain from \eqref{eq:rec_ABC} and Theorem~\ref{th:trasform_of_ternary} the following corollaries.
\begin{corollary} If $2\leq n\leq k$, then
	\begin{equation*}\label{eq:fib_diag}
		a_n^k=6 a_{n-1}^{k-1}-4  a_{n-2}^{k-2}.
	\end{equation*}
\end{corollary}
\begin{corollary} 
	The trinomial transform sequence $(b_n)$ of the Fibonacci sequence is the binary sequence  $b_n=6b_{n-1}-4b_{n-2}$  with initial elements $b_0=0$, $b_1=2$. (The main diagonal of Table~\ref{tab:fib}.) 
\end{corollary}

In OEIS the trinomial transform of the Fibonacci numbers is the sequence \seqnum{A082761} $(1, 4, 20, 104,\ldots)$, which can be seen in Table~\ref{tab:fib} as the second diagonal.

We obtain the following corollaries from the theorems of the previous section.

\begin{corollary} If $2\leq n\leq k$, then 
	\begin{eqnarray*}
		F^{[2^n]}_{k+n}&=& 2F^{[2^{n-1}]}_{k+n-1} + 4 F^{[2^{n-2}]}_{k+n-2},\\
		F^{[2^n]}_{k+n}&=& 6F^{[2^{n-1}]}_{k+n-2} -4 F^{[2^{n-2}]}_{k+n-4}.
	\end{eqnarray*}
\end{corollary}

\begin{corollary} For the sum and alternating sum of columns, we have
	\begin{eqnarray*}
		s_n&=& \sum_{i=0}^{n}F^{[2^n]}_{i+n}=\sum_{i=0}^{2n} {\binomtwos{n}{i}} F_{i}=\sum_{i=0}^{n} \sum_{k=0}^{2i} \binomtwo{i}{k}F_{n+k-i},
		\\
		\bar{s}_n&=& \sum_{i=0}^{n}(-1)^iF^{[2^n]}_{i+n}=\sum_{i=0}^{n} (-1)^i \sum_{k=0}^{2i} \binomtwo{i}{k}F_{n+k-i}.
	\end{eqnarray*}
	Moreover, if $n\ge4$, then
	\begin{eqnarray*}\label{eq:fib_nat}
		s_k&=&7s_{k-1}-9s_{k-2}-2s_{k-3}+4s_{k-4},\\
		\bar{s}_k&=&-5\bar{s}_{k-1}+3\bar{s}_{k-2}+10\bar{s}_{k-3}+4\bar{s}_{k-4}.
	\end{eqnarray*}
	with initial values $s_0=0$, $s_1=3$,  $s_2=17$, $s_3=92$, and $\bar{s}_0=0$, $\bar{s}_1=-1$,  $\bar{s}_2=9$, $\bar{s}_3=-48$.
\end{corollary}

\subsection{Tribonacci sequence}

Let the base sequence $a_0^k$ be the Tribonacci sequence (\seqnum{A000073}), defined by $t_k=t_{k-1}+t_{k-2}+t_{k-3}$, $k\geq3$ with initial values $t_0=0$, $t_1=0$, $t_2=1$. So let $a_0^k=t_k$ and $\alpha=\beta=\gamma=1$. The Tribonacci trinomial transform triangle is depicted in Table~\ref{tab:tribo}. 

\begin{table}[htb]
	\centering
	\setlength{\tabcolsep}{3pt}
	\begin{tabular}{|c|cccccccccc|}\hline
		& 0           & 1          & 2          & 3           & 4            & 5            & 6             & 7             & 8              & 9                          \\ \hline
		0         & 0          & 0          & 1           & 1           & 2            & 4             & 7              & 13             & 24              & 44      \\
		1         &            & 1          & 2           & 4           & 7            & 13            & 24             & 44             & 81              & 149     \\
		2         &            &            & 7           & 13          & 24           & 44            & 81             & 149            & 274             & 504     \\
		3         &            &            &             & 44          & 81           & 149           & 274            & 504            & 927             & 1705    \\
		4         &            &            &             &             & 274          & 504           & 927            & 1705           & 3136            & 5768    \\
		5         &            &            &             &             &              & 1705          & 3136           & 5768           & 10609           & 19513   \\
		6         &            &            &             &             &              &               & 10609          & 19513          & 35890           & 66012   \\
		7         &            &            &             &             &              &               &                & 66012          & 121415          & 223317  \\
		8         &            &            &             &             &              &               &                &                & 410744          & 755476  \\
		9         &            &            &             &             &              &               &                &                &                 & 2555757 \\ \hline
		$s_k$   & \textbf{0} & \textbf{1} & \textbf{10} & \textbf{62} & \textbf{388} & \textbf{2419} & \textbf{15058} & \textbf{93708} & \textbf{583100} & \textbf{3628245}\\
		$\bar{s}_k$   & 0 & $-1$ & 6 & $-34$ & 212 & $-1315$ & 8190 & $-50948$ & 317036 & $-1972637$ \\ \hline
	\end{tabular}
	\caption{Tribonacci trinomial transform triangle}
	\label{tab:tribo}
\end{table}

Easy to see that $a_n^k= a_{n-1}^{k+2} = \cdots = a_{0}^{k+2n}=t_{k+2n}$, and from it with the statements in the previous section, we have the following corollaries.

\begin{corollary} In the case $3\leq n\leq k$, we obtain
	\begin{eqnarray*}
		a_n^k&=&7 a_{n-1}^{k-1}-5  a_{n-2}^{k-2}+a_{n-3}^{k-3},\\
		a_n^k&=&3 a_{n-1}^k+   a_{n-2}^k+  a_{n-3}^k
	\end{eqnarray*}
	Moreover, the terms of sequence $(a_n^k)_{n=0}^{k}$ are every second terms of the Tribonacci numbers, so that $a_n^n=t_{3n}$.
\end{corollary}

\begin{corollary} 
	The trinomial transform sequence $(b_n)$ of the Tribonacci sequence is the ternary sequence  $b_n=6b_{n-1}-4b_{n-2}+b_{n-3}$ with initial terms $b_0=0$, $b_1=1$, $b_2=7$. (The main diagonal of Table~\ref{tab:tribo}.) 
\end{corollary}

In OEIS the trinomial transform of the Tribonacci numbers is the sequence \seqnum{A192806} $(1, 1, 4, 24, 149, 927,\ldots)$, which is the third diagonal in Table~\ref{tab:fib}  with an additional first term.

\begin{corollary}  For the sum and alternating sum of columns, we have
	\begin{eqnarray*}
		s_n&=& \sum_{i=0}^{n}t_{n+2i}=\sum_{i=0}^{2n} {\binomtwos{n}{i}} t_{i}=\sum_{i=0}^{n} \sum_{k=0}^{2i} \binomtwo{i}{k}t_{n+k-i},
		\\
		\bar{s}_n&=& \sum_{i=0}^{n}(-1)^i t_{n+2i}=\sum_{i=0}^{n} (-1)^i \sum_{k=0}^{2i} \binomtwo{i}{k}t_{n+k-i}.
	\end{eqnarray*}
	In addition,	
	\begin{eqnarray*}
		s_n&=&8s_{n-1}-11s_{n-2}-3s_{n-4}+4s_{n-5}-s_{n-6} \qquad (n\ge6),\\
		\bar{s}_n&=&-6\bar{s}_{n-1}+3\bar{s}_{n-2}+12\bar{s}_{n-2}+13\bar{s}_{n-4}+6\bar{s}_{n-5}+\bar{s}_{n-6} \qquad (n\ge6),
	\end{eqnarray*}
	with initial values $s_0=0$, $s_1=1$,  $s_2=10$, $s_3=62$,  $s_4=388$, $s_5=2419$, and $\bar{s}_0=0$, $\bar{s}_1=-1$,  $\bar{s}_2=6$, $\bar{s}_3=-34$,  $\bar{s}_4=212$, $\bar{s}_5=-1315$.
\end{corollary}

\subsection{The constant sequence 1}

In this subsection, we give Table~\ref{tab:const} generated by the constant sequence 1  (\seqnum{A000012}) and the expressions that prove Theorem~\ref{th:sum_trinom_triangle}. 

\begin{table}[h]
	\centering
	\setlength{\tabcolsep}{3pt}
	\begin{tabular}{|c|cccccccccc|}\hline
		& 0           & 1          & 2          & 3           & 4            & 5            & 6             & 7             & 8              & 9                          \\ \hline
		0         & 1          & 1          & 1           & 1           & 1            & 1            & 1             & 1             & 1             & 1              \\
		1         &            & 3          & 3           & 3           & 3            & 3            & 3             & 3             & 3             & 3              \\
		2         &            &            & 9           & 9           & 9            & 9            & 9             & 9             & 9             & 9              \\
		3         &            &            &             & 27          & 27           & 27           & 27            & 27            & 27            & 27             \\
		4         &            &            &             &             & 81           & 81           & 81            & 81            & 81            & 81             \\
		5         &            &            &             &             &              & 243          & 243           & 243           & 243           & 243            \\
		6         &            &            &             &             &              &              & 729           & 729           & 729           & 729            \\
		7         &            &            &             &             &              &              &               & 2187          & 2187          & 2187           \\
		8         &            &            &             &             &              &              &               &               & 6561          & 6561           \\
		9         &            &            &             &             &              &              &               &               &               & 19683                         \\\hline
		$s_k$       & \textbf{1} & \textbf{4} & \textbf{13} & \textbf{40} & \textbf{121} & \textbf{364} & \textbf{1093} & \textbf{3280} & \textbf{9841} & \textbf{29524}\\
		$\bar{s}_k$   & 1& $-2$& 7& $-20$& 61& $-182$& 547& $-1640$& 4921& $-14762$\\ \hline
	\end{tabular}
	\caption{Trinomial transform triangle generated by the constant sequence 1} 
	\label{tab:const}
\end{table}

One can easily see that $a_n^k=3^n$ (\seqnum{A000244}). It implies that the trinomial transform sequence of the constant sequence 1  is the sequence $3^n$. Moreover, the columns also form geometric sequences with common ratio $3$. Thus 

\begin{eqnarray*}
	s_{n} &=& \sum_{i=0}^{2n} {\binomtwos{n}{i}} =\sum_{i=0}^{n} \sum_{k=0}^{2i} \binomtwo{i}{k}=\sum_{i=0}^{n}3^i=\frac{3^{n+1}-1}{2}\quad \text{(\seqnum{A003462})},\label{eq:sum_const}\\
	\bar{s}_n &=& \sum_{i=0}^{n}  \sum_{k=0}^{2i} \binomtwo{i}{k}(-1)^i=\sum_{i=0}^{n}(-3)^i=\frac{3(-3)^n+1}{4}\quad \text{(\seqnum{A014983})},\label{eq:altsum_const}
\end{eqnarray*}

\subsection{Natural numbers}

Let $a_0^k=k$, the non-negative integers (\seqnum{A001477}). We obtain that $a_n^k=k3^n$, and the trinomial transform of the natural sequence is the sequence $n3^n$ (\seqnum{A036290}).

In this case, the columns also form geometric sequences with common ratio $3$ (see Table~\ref{tab:natur}). Thus 
\begin{eqnarray*}
	s_{n} &=& \sum_{i=0}^{2n} i {\binomtwos{n}{i}} =\sum_{i=0}^{n} \sum_{k=0}^{2i} n \binomtwo{i}{k}= \sum_{i=0}^{n}n3^i=n\frac{3^{n+1}-1}{2}, \label{eq:sum_natural}\\
	\bar{s}_n &=&\sum_{i=0}^{n}  \sum_{k=0}^{2i} \binomtwo{i}{k}(-1)^in = \sum_{i=0}^{n}n(-3)^i= n\frac{3(-3)^n+1}{4}. \label{eq:altsum_natural}
\end{eqnarray*}

\begin{table}[h]
	\centering
	\setlength{\tabcolsep}{3pt}
	\begin{tabular}{|c|cccccccccc|}\hline
		& 0           & 1          & 2          & 3           & 4            & 5            & 6             & 7             & 8              & 9                          \\ \hline
		0         & 0          & 1          & 2           & 3            & 4            & 5             & 6             & 7              & 8              & 9               \\
		1         &            & 3          & 6           & 9            & 12           & 15            & 18            & 21             & 24             & 27              \\
		2         &            &            & 18          & 27           & 36           & 45            & 54            & 63             & 72             & 81              \\
		3         &            &            &             & 81           & 108          & 135           & 162           & 189            & 216            & 243             \\
		4         &            &            &             &              & 324          & 405           & 486           & 567            & 648            & 729             \\
		5         &            &            &             &              &              & 1215          & 1458          & 1701           & 1944           & 2187            \\
		6         &            &            &             &              &              &               & 4374          & 5103           & 5832           & 6561            \\
		7         &            &            &             &              &              &               &               & 15309          & 17496          & 19683           \\
		8         &            &            &             &              &              &               &               &                & 52488          & 59049           \\
		9         &            &            &             &              &              &               &               &                &                & 177147          \\ \hline
		$s_k$   & \textbf{0} & \textbf{4} & \textbf{26} & \textbf{120} & \textbf{484} & \textbf{1820} & \textbf{6558} & \textbf{22960} & \textbf{78728} & \textbf{265716}\\
		$\bar{s}_k$ &  0& $-2$& 14& $-60$& 244& $-910$& 3282& $-11480$& 39368& $-132858$ \\ \hline
	\end{tabular}
	\caption{Trinomial transform triangle generated by natural numbers}
	\label{tab:natur}
\end{table}

\subsection{Two other sequences}

In this subsection, we give two cases without tables, whose trinomial transform sequences are the all 1's sequence and the natural numbers mentioned in the previous subsections. We gain the results with very easy calculations.

First, let  $a_0^k=(-1)^k$ (\seqnum{A033999}). Then $a_n^k=(-1)^{n+k}$ and $b_n=1$. 
\begin{eqnarray*}
	s_{n} &=& \sum_{i=0}^{2n} (-1)^i {\binomtwos{n}{i}} =\sum_{i=0}^{n} \sum_{k=0}^{2i} (-1)^{n+k-i} \binomtwo{i}{k}= \sum_{i=0}^{n}(-1)^{n+i}=
	\begin{cases}
		0,& \text{if $n$ is odd};\\
		1,& \text{if $n$ is even}.
	\end{cases}  \\
	\bar{s}_n &=&\sum_{i=0}^{n}  \sum_{k=0}^{2i} \binomtwo{i}{k}(-1)^i = \sum_{i=0}^{n}1= n+1. 
\end{eqnarray*}

Second, let  $a_0^k=(-1)^kk$ (\seqnum{A038608}). Then $a_n^k=(-1)^{n+k}k$ and $b_n=n$. 
\begin{eqnarray*}
	s_{n} &=& \sum_{i=0}^{2n} (-1)^i i{\binomtwos{n}{i}} =\sum_{i=0}^{n} \sum_{k=0}^{2i} (-1)^{n+k-i} n \binomtwo{i}{k}= \sum_{i=0}^{n}(-1)^{n+i}n\\
	&=&
	\begin{cases}
		0,& \text{if $n$ is odd or $n=0$};\\
		n,& \text{otherwise}.
	\end{cases}  \\
	\bar{s}_n &=&n\sum_{i=0}^{n}  \sum_{k=0}^{2i} \binomtwo{i}{k}(-1)^i = \sum_{i=0}^{n}n= n^2+n \qquad (\seqnum{A002378}). 
\end{eqnarray*}


\begin{thebibliography}{2018}
	
	\bibitem{Andrews} G. Andrews, {Euler's 'exemplum memorabile inductionis fallacis' and $q$-Trinomial Coefficients}, {\it J. Amer. Math. Soc.} \textbf{3} (1990), 653--669. 
	
	\bibitem{ANVI} G. Anatriello and G. Vincenzi, {Tribonacci-like sequences and generalized Pascal's pyramids}, {\it Internat. J. Math. Ed. Sci. Tech.} \textbf{45} (2014), 1220--1232.
	
	\bibitem{Barbero1} S. Barbero, U. Cerruti, and N. Murru, Trasforming recurrent sequences by using Binomial and Invert operators, \emph{J. Integer Seq.} \textbf{13} (2010), \href{http://www.cs.uwaterloo.ca/journals/JIS/VOL13/Barbero/barbero5.pdf}{Article 10.7.7}.
	
	\bibitem{Barbero} S. Barbero, U. Cerruti, and  N. Murru,  A generalization of the binomial interpolated operator and its action on linear recurrent sequences,  \textit{J. Integer Seq.} \textbf{13} (2010), \href{https://cs.uwaterloo.ca/journals/JIS/VOL13/Barbero2/barbero7.html}{Article 10.9.7.}
	
	\bibitem{Bel2008} H. Belbachir, S. Bouroubi, and A. Khelladi, {Connection between ordinary multinomials, Fibonacci numbers, Bell polynomials and discrete uniform distribution}, {\it Ann. Math. Inform.} \textbf{35} (2008), 21--30.
	
	\bibitem{Nemeth_binom} L. N\'emeth, {On the binomial interpolated triangles}, {\it J. Integer Seq.} \textbf{20} (2017),  
	\href{https://cs.uwaterloo.ca/journals/JIS/VOL20/Nemeth/nemeth2.html} {Article 17.7.8}.
	
	\bibitem{Pan} J. Pan, Some properties of the multiple binomial transform and the Hankel transform of shifted sequences, \textit{J. Integer Seq.} \textbf{14} (2011), \href{https://cs.uwaterloo.ca/journals/JIS/VOL14/Pan/pan12.html} {Article 11.3.4}.
	
	\bibitem{Sloane} N. J. A. Sloane, The On-Line Encyclopedia of Integer Sequences, \href{http://oeis.org}{http://oeis.org}.
	
	\bibitem{sho} T. N. Shorey and R. Tijdeman, \textit{Exponential Diophantine Equations}, Cambridge University Press, 1986.
	
	\bibitem{Spivey} M. Z. Spivey and L. L. Steil,  The $k$-binomial transforms and the Hankel transform,  \textit{J. Integer Seq.} \textbf{9} (2006),  \href{https://cs.uwaterloo.ca/journals/JIS/VOL9/Spivey/spivey7.html}{Article 06.1.1}.
	
\end{thebibliography}
\end{document}